\documentclass[reqno]{amsart}

\usepackage{amssymb,amsfonts,amstext,amsthm,hyperref,cleveref,xcolor}

\usepackage{graphics}
\usepackage{graphicx}
\usepackage{epstopdf}
\usepackage{indentfirst}
\usepackage{float}
\usepackage[numbers,sort&compress]{natbib}
\bibpunct[, ]{[}{]}{,}{n}{,}{,}
\makeatletter
\def\NAT@def@citea{\def\@citea{\NAT@separator}}
\makeatother
\theoremstyle{plain}
\newtheorem{theorem}{Theorem}[section]
\newtheorem{lemma}[theorem]{Lemma}

\crefrangeformat{equation}{#3(#1)#4--#5(#2)#6}
\crefname{enumi}{\unskip}{\unskip}

\theoremstyle{definition}

\newtheorem{example}[theorem]{Example}
\newtheorem{remark}[theorem]{Remark}

\begin{document}

\title[$(k+1)$-potent Matrices]{$(k+1)$-potent Matrices in triangular matrix Groups and Incidence Algebras of Finite Posets}

\author{Ivan Gargate}
\address{UTFPR, Campus Pato Branco, Rua Via do Conhecimento km 01, 85503-390 Pato Branco, PR, Brazil}
\email{ivangargate@utfpr.edu.br}

\author{Michael Gargate}
\address{UTFPR, Campus Pato Branco, Rua Via do Conhecimento km 01, 85503-390 Pato Branco, PR, Brazil}
\email{michaelgargate@utfpr.edu.br}

\begin{abstract}
Let $\mathbb{K}$ be a field such that  $char(\mathbb{K})\nmid k$ and $char(\mathbb{K})\nmid k+1$.  We describe all $(k+1)$-potent matrices over the group of  upper triangular matrix. In the case that $\mathbb{K}$ is a finite field we show how to compute the number of these elements in triangular matrix groups and use this formula to compute the number of $(k+1)$-potent elements in the Incidence Algebra $\mathcal{I}(X,\mathbb{K})$ where $X$ is a finite poset.

\end{abstract}
\maketitle

\section{Introduction}
Let $G$ be a group. An element $g\in G$ is called $(k+1)$-potent if satisfies $g^{k+1}=g$ for $k\in \mathbb{N}, k\geq 1$. If $k=1$ the element $g$ is called of idempotent. In the case that $G$ is the group of matrices, there are several authors
who have works with idempotent elements , for instance, expressing various matrices as a difference, a sum and product of some idempotents, see \cite{Har, Har2, Slowik3, Wu, Bas, Bot, Erdo}. Recently Hou in \cite{Hou}  describe the necessary and sufficient conditions on the infinite and finite upper triangular matrices over a unitary ring to be idempotents with only zeros and ones into  the main diagonal. 

In the case of $k$-potent matrix the authors in \cite{Gargate} generalize the article of Wu \cite{Wu} and show the conditions by which  a complex matrix can expressed as a sum of $k$-potent matrices.
 
 In this article, we investigate how to construct $(k+1)$-potent matrices on upper triangular matrices over $\mathbb{K}$ generalizing the recent work of Hou \cite{Hou} on idempotent triangular matrices. Also, we described the necessary and sufficient conditions on infinite and finite upper triangular matrices over a unitary ring to be $(k+1)$-potent matrices with set elements $\{0, 1, \omega, \omega^2, \ldots, \omega^{k-1}\}$ on the diagonal, where $\omega$ is a $kth$ root of unity and we present a formula to compute all $(k+1)$-potent matrices in the case that $\mathbb{K}$ is a finite field. Furthermore,  we present various formulas to compute the number of $(k+1)$-potent elements on  the incidence algebra $\mathcal{I}(X,\mathbb{K})$ with $X$ be a finite poset. By the lemma 1.18 of Gubareni-Hazewinkel \cite{gub1} we focus only to compute the number of $(k+1)$-potent elements over some special posets: Star, Rhombuses and Y. With the combination of these specially posets we can generate any finite poset, but, an unique formula to compute the number of $(k+1)$-potent elements in general finite posets is complicated by the multiple forms of a finite poset. Fortunately in this article we shown an road to compute the number of these elements.

Our main result is the following Theorem:

\begin{theorem}\label{th1} Assume that $\mathbb{K}$ is a field such that $char(\mathbb{K}) \nmid k$ and $char(\mathbb{K})\nmid k+1$. Let $G$ be either the matrix
group $T_{\infty}( \mathbb{K})$ or $T_n(\mathbb{K})$ for some $n \in \mathbb{N}$. Then an element $A$ of the matrix group $G$  with $a_{ii}\in\{0,1,\omega,\omega^{2},\ldots,\omega^{k-1}\}$ where $\omega$ kth root of unity in $\mathbb{K}$ is an
$(k+1)$-potent matrix  if and only if A is described by the following statements.
\begin{itemize}
\item[(i)] For all $i$, $a_{ii}\in\{0, 1, \omega, \omega^2, \ldots, \omega^{k-1}\}$.
\item[(ii)] For $i < j$, if $a_{ii} = a_{jj}$, then $a_{ij}$ equals to
\begin{equation}\label{eqth}
a_{ij}=\!\!\left\{\!\!\!\begin{array}{lr} 
0\! &\!\text{if} \ \ \! j\!=\!i\!+\!1 \\

\!\!
 -\displaystyle\frac{1}{k}\left[\displaystyle\sum_{s=2}^{k+1}\sum_{i<i_1\leq \! \cdots \leq i_{s-1}<j}(k\!+\!2\!-\!s)(a_{ii})^ {k+1-s}a_{i,i_1}a_{i_1,i_2}\cdots a_{i_{s-1},j}\right]\!
&\!\text{if}\ \ \! j\! >\! i\! +\!1. 
 \end{array}\right.
\end{equation}
\item[(iii)] For $i < j$, if $a_{ii} \not = a_{jj}$, then $a_{ij}$ can be chosen arbitrarily.
\end{itemize}

\end{theorem}
 
In the section 3, using the above theorem, we will prove the following formula of $(k+1)$-potent matrices in upper triangular matrix groups

\begin{theorem}\label{th2} Let $\mathbb{K}$ be a finite field such that $char(\mathbb{K}) \nmid k$, $char(\mathbb{K})\nmid k+1$ and $|\mathbb{K}|=q$. Then the total number of $ n\times n$ upper triangular $(k+1)$-potent matrices with elements of the set  $D=\{0,1,\omega, \omega^2,\ldots.\omega^{k-1}\}$ on
the diagonal where $\omega$ is a kth-root of unity in the matrix ring $T_n( \mathbb{K})$ is equal to
$$ \displaystyle \sum_{\substack{n_0+\cdots+ n_k=n \\ 0\leq n_i}} \binom{n}{n_0n_1n_2\cdots n_k}\cdot q^{\displaystyle \sum_{\substack{0\leq i<j \leq k }}n_i\cdot n_j}, \ where \ \binom{n}{n_0\cdots n_k}=\frac{n!}{n_0!\cdot n_1\cdots n_k!}.
$$
\end{theorem}

In the section 4, for Incidence Algebras, we assume that there are $s$ $(k+1)$-potent elements in $\mathbb{K}$, in this case we have for Star Poset

\begin{theorem}[The Star Poset]\label{teo1} Let $\mathbb{K}$ be a field of be any field such that $char\mathbb{K}\nmid  k$, $char \mathbb{K}\nmid (k+1)$ and $|\mathbb{K}|=q$. Suppose that there are $s$ elements in $\mathbb{K}$ such that $x^{k+1}=x$, with $k$ an positive integer. Consider the poset $X=\{x_0,x_1, x_2, \ldots, x_n,y_1, y_2,\ldots,y_m\}$  with the relations
\begin{itemize}
\item $x_0\leq x_i $ for $i=1,2, \ldots, n$, and $x_i\leq x_j$ for $i\leq j$ with $i,j=1, 2, \ldots, n$,
\item $x_0\leq y_i $ for $i=1,2, \ldots, m$, and $y_i\leq y_j$ for $i\leq j$ with $i,j=1, 2, \ldots, m$.
\end{itemize}

Denote by $S=\{x_0,x_1,\cdots,x_n\}$ the subposet of $X$ and by $\mathcal{N}(X,k+1,\mathbb{K})$ the number of $(k+1)$-potent elements in the incidence algebra $\mathcal{I}(X, \mathbb{K})$. Then
$$\mathcal{N}(X,k+1,\mathbb{K})=\mathcal{N}(n+1,k+1,\mathbb{K})\cdot P(m,k+1),$$
where
$$P(m,k+1)=\sum_{\substack{m_1+m_2+\cdots +m_s=m \\ 0\leq m_i}}\binom{m}{m_1m_2\cdots m_s} q^{\Delta(d_S)}\cdot q^{ \sum^s_{i=2}m_i}$$
and $\mathcal{N}(n+1,k+1,\mathbb{K})$ is the number of $(k+1)$-potent elements in the upper triangular group $UT_{(n+1)}(\mathbb{K})$. In general, if consider the poset $X=\bigcup^{r}_{i=1}X_{i}$ such that $x_0$ is the minimal element with $\bigcap^{r}_{i=1}X_{i}=\{x_0\}$ and each one $X_{i}$ is an interval with length $m_i+1$,then, the number of $(k+1)$-potent elements   in the incidence algebra $\mathcal{I}(X,\mathbb{K})$ is
$$\mathcal{N}(X,k+1,\mathbb{K})=\mathcal{N}(X_1,k+1,\mathbb{K})\cdot \prod_{t=2}^{r} P(m_t,k+1).$$
\end{theorem}

We also consider the case of Rhombuses poset:

\begin{theorem}[The Rhombuses poset]\label{teo2} Let $\mathbb{K}$ be  any field such that $char\mathbb{K}\nmid  k$, $char \mathbb{K}\nmid (k+1)$ and $\left|\mathbb{K}\right|=q$. Consider the Rhombuses poset 
$$X=\{x_0,x_1,x_2, \cdots ,x_n,x_{n+1},y_1,y_2, \cdots, y_m\},$$ 
with the relations:
\begin{itemize}
    \item $x_0 \leq x_1 \leq \cdots \leq x_n \leq x_{n+1}$ and $ y_1 \leq y_2 \leq \cdots \leq y_m$,
    \item $x_0\leq y_1$ and $y_m \leq x_{n+1}$.
    \end{itemize} 
    Suppose that there are $s$  solutions of the equation $x^{k+1}=x$ in $\mathbb{K}$, then we have:

$$ \displaystyle \sum_{\substack{n_1+n_2+\cdots+n_s=n, \\ m_1+m_2+\cdots+m_s=m, \\ 0\leq n_i, 0\leq m_i,\ i=1,2,\cdots,s}} s \cdot  \binom{n}{n_1n_2\cdots n_s}\binom{m}{m_1m_2\cdots m_s}\cdot q^{\displaystyle\sum_{1\leq i< j \leq s} n_i\cdot m_j+m_i\cdot m_j}\times $$ 
$$ \ \ \ \ \ \ \ \ \ \ \ \times\mathcal{F}(n_1,\cdots,n_s,m_1,\cdots,m_s),$$
$(k+1)$-potent elements in $\mathcal{I}(X,\mathbb{K})$, where
$$\mathcal{F}(n_1,\cdots,n_s,m_1,\cdots,m_s)=q^{2(n-n_1+m-m_1)}+ \displaystyle \sum_{j=2}^s q^{2(n+m)+1-(n_1+n_j+m_1+m_j)}.$$

\end{theorem}
Finally, for the $Y$ poset
\begin{theorem}[The $Y$ Poset]\label{teo3} Let $\mathbb{K}$ be any field such that $char\mathbb{K}\nmid  k$, $char\mathbb{K} \nmid (k+1)$ and $\left|\mathbb{K}\right|=q$. Consider the poset $Y=\{r_1,r_2, \cdots, r_n, s_1,s_2 \cdots, s_m, t_1,t_2, \cdots t_l\}$ with the relations:
\begin{itemize}
    \item $r_1,\leq r_2 \leq \cdots \leq  r_n$, $s_1 \leq s_2 \leq \cdots \leq s_m$ and $t_1 \leq t_2 \leq \cdots\leq  t_l$,
    \item $r_n \leq s_i$ with $i=1,2,\cdots, m$ and $r_n\leq t_j$ with $j=1,2,\cdots, l$.
\end{itemize}
Suppose that there are $s$ solutions of the equation  $x^{k+1}=x$ in $\mathbb{K}$, then we have:

$$\displaystyle \sum_{\substack{n_1+n_2+\cdots+n_s=n, \\ m_1+m_2+\cdots+m_s=m, \\ l_1+l_2+\cdots+l_s=l, \\ 0\leq n_i,m_i,l_i, \ i=1,2,\cdots, s}}\binom{n}{n_1n_2\cdots n_s}\binom{m}{m_1m_2\cdots m_s}\binom{l}{l_1l_2\cdots l_s} \times \ \ \ \ \ \ \ $$ 
$$\ \ \ \ \ \ \ \ \ \ \ \ \ \ \times \mathcal{F}(n_1,\cdots,n_s,m_1,\cdots,m_s,l_1,\cdots,l_s),$$
$(k+1)$-potent elements in $\mathcal{I}(X,\mathbb{K})$, where
$$\mathcal{F}(n_1,\cdots\!,n_s,m_1,\!\cdots\!,m_s,l_1,\!\cdots\!,l_s)\!=\!q^{\displaystyle\sum_{1\leq i<j\leq s}\!\! (n_in_j\!+\!m_im_j\!+\!l_il_j)}\!\! \cdot q^{\displaystyle \sum_{1\leq i\neq j \leq s}\!\! n_i(m_j\!+\!l_j)}.$$
\end{theorem}

In the section 5 we present an important remark and we compare our approach with the result obtain by Slowik in \cite{Slowik}.

In the section 6 we present various tables with some result about the number of $(k+1)$-potent elements.

\section{$k$-Potent in triangular matrix group}\label{intro}

In this section we consider $T_n(\mathbb{K})$ the group of upper triangular matrices of dimension $n$ and $T_{\infty}(\mathbb{K})$ the group of infinite upper triangular matrices.  Let $k \geq 1$ be an integer. If
a square matrix $A$ satisfies $A^{k+1} = A$, then $A$ is said to be $(k+1)$-potent. For $k = 1$, $A$ is said idempotent matrix. We assume that there is $\omega$ a $k$th non-trivial root of unity in $\mathbb{K}$. 

We start our considerations we notice the following property.

\begin{remark}\label{remark1} Assume that $\mathbb{K}$ is a field, $G=T_{\infty}(\mathbb{K})$ or $G=T_n(\mathbb{K})$ for some $n\in\mathbb{N}$. If $A\in G$ is a block matrix such that
$$A=\left[ \begin{array}{cccc}
B_{11} &B_{12}& B_{13}& \cdots \\ 
&B_{22}&B_{23} & \cdots\\
 && B_{33}& \cdots  \\
 &&&\ddots
\end{array}
\right]
$$
where $B_{ii}$ are square  matrices and $A$ is a $(k+1)$-potent matrix, then for all $i$, the matrices $B_{ii}$ are $(k+1)$-potents matrices as well.
\end{remark}
\begin{proof} Since $A$ is a $(k+1)$-potent matrix of $G$, we have
$$A^{k+1}=\left[ \begin{array}{cccc}
B_{11}^{k+1} & *& *& \cdots \\ 
&B_{22}^{k+1}&* & \cdots\\
 && B_{33}^{k+1}& \cdots  \\
 &&&\ddots
\end{array}
\right]=\left[ \begin{array}{cccc}
B_{11} &*& *& \cdots \\ 
&B_{22}&* & \cdots\\
 && B_{33}& \cdots  \\
 &&&\ddots
\end{array}
\right]=A,
$$
and we obtain $B_{ii}^{k+1}=B_{ii}$ for all $i$ by comparing entries of the diagonal position in the matrix equality above.
\end{proof}

Before we prove our main result, in Slowik \cite{Slowik} we have the following information about the powers of triangular matrices.
\begin{lemma}[Lemma 2.1 Slowik \cite{Slowik}] Let $\mathbb{K}$ be an arbitrary field and let $n\geq 3$. If $A\in T_n(\mathbb{K})$ is of the form
\begin{equation}\label{eq11}
A= \left[ \begin{array}{ccc}
r &a & x \\
0  & u & b \\ 
0  & 0 & t
\end{array}
\right]
\end{equation}
where $r,t \in \mathbb{K}$ and $u \in T_{n-2}(\mathbb{K})$, then for any $k \in \mathbb{N}$ we have

\begin{equation}\label{eq22}
A^{k+1}= \left[ \begin{array}{ccc}
r^{k+1} &a' & x' \\ 
0  & u^{k+1} & b' \\
0  & 0 & t^{k+1}
\end{array}
\right]
\end{equation}
where
\begin{equation}\label{aa}
a' = a\cdot (r^{k} \cdot e + r^{k-1}\cdot u + r^{k-2}\cdot u^2 +\cdots +r \cdot u^{k-1} + u^{k})
\end{equation}
\begin{equation}\label{bb}
b' = (u^{k} + u^{k-1}\cdot t + u^{k-2}\cdot t^2 +\cdots +u \cdot t^{k-1} + e\cdot t^{k}) \cdot b
\end{equation}

\begin{equation}\label{xx}
\begin{array}{ccl}
x' & =&x\cdot (r^k+r^{k-1}\cdot t+\cdots+ r\cdot t^{k-1}+t^k)  \\
 &&\\
     & & +a\cdot [(r^{k-1}+r^{k-2}\cdot t+\cdots+r\cdot t^{k-2}+t^{k-1})\cdot e  \\
     &&\\
     & & +(r^{k-2}+r^{k-3}\cdot t+\cdots+r^j\cdot t^{k-3}+t^{k-2})\cdot u\\
     &&\\
     & &+ (r^{k-3}+r^{k-4}\cdot t+\cdots+r\cdot t^{k-4}+t^{k-3})\cdot u^2 +\cdots\\
     &&\\
     & & \cdots +(r+t)\cdot u^{k-2}+u^{k-1}]\cdot b.
     \end{array}
\end{equation}
\end{lemma}

From this Lemma we obtain that
\begin{equation}\label{eq1}
\left[\! \begin{array}{c|c}
r \!&\!a\\ \hline
0  \!&\! u
\end{array}
\!\right]^{k+1}\!\!\!\!\!=\!\left[\! \begin{array}{c|c}
r^{k+1} \!\!&\!\! a\cdot (r^{k} \cdot e + r^{k-1}\cdot u +r^{k-2}\cdot u^2 +\cdots +r \cdot u^{k-1} + u^{k})\\ \hline
0 \!\! &\!\! u^{k+1}
\end{array}
\!\right]
\end{equation}
and
\begin{equation}\label{eq2}
\left[\! \begin{array}{c|c}
u \!&\!b\\ \hline
0 \! & \!t
\end{array}
\!\right]^{k+1}\!\!\!\!\!=\!\left[ \begin{array}{c|c}
u^{k+1} \!&\! (u^{k} + u^{k-1}\cdot t + u^{k-2}\cdot t^2 +\cdots +u \cdot t^{k-1} + e\cdot t^{k}) \cdot b\\ \hline
0  \!&\! t^{k+1}
\end{array}
\!\right].
\end{equation}

Now, we can prove our first main result.




\begin{proof}[Proof of Theorem \ref{th1}] 
Suppose that $A=\sum_{i\leq j}E_{ij}$ is a $(k+1)$-potent element of $T_n(\mathbb{K})$. Let $\omega$ be a $kth$ root of unity and let $A(m,i)$ be the submatrix of $A$ defined as
$$A(m,i)=\left[ \begin{array}{cccc}
a_{ii} &a_{i,i+1} & \cdots & a_{i,i+m}\\ 
&a_{i+1,i+1}&\cdots&a_{i+1,i+m}\\
 & & \ddots & \vdots\\
 &&&a_{i+m,i+m}
\end{array}
\right]
$$
where $a_{ii}\in\{0, 1, \omega, \omega^2, \ldots, \omega^{k-1}\}$.

From Remark (\ref{remark1}) one can see that $A$ is a $(k+1)$-potent matrix if and only if $A(m,i)$ is $(k+1)$-potent for all $m$ and $i$.

Since $A$ is an $(k+1)$-potent matrix with $a_{ii}\in\{0, 1, \omega, \omega^2, \ldots, \omega^{k-1}\}$, we need to prove that  (ii) and (iii) given in the Theorem \ref{th1} hold.

Consider the elements $a_{ii}$ and $a_{jj}$ of the matrix $A$ with $i<j$. Then we have that $j-i=1$ or $j-i>1$ and we use induction on $j-i=m$.

Assume that $j-1=1$, and denote for $a_{ii}=r$ and $a_{i+1,i+1}=t$ then from $A(1,i)^{k+1}=A(1,i)$ and the equation (\ref{eq1}) we obtain
\begin{equation}\label{eqm=1}
A(1,i)^{k+1}=\left[ \begin{array}{c|c}
r^{k+1} &a_{i,i+1}\cdot (r^{k} + r^{k-1}\cdot t + r^{k-2}\cdot t^2 +\cdots +r \cdot t^{k-1} + t^{k})\\ \hline
0  & t^{k+1}
\end{array}
\right].
\end{equation}

Thus
\begin{itemize}
\item If $r\not=t$, we have that
\begin{equation}\label{eqk}
\sum_{p=0}^{k}r^p t^{k-p}=
\left( r^{k+1}-t^{k+1}\right)\cdot (r-t)^{-1},
\end{equation}
then
$$
\begin{array}{rcl}
\!\!a_{i,i+1}\cdot (\!r^{k} + \!r^{k-1}\cdot t +\! r^{k-2}\cdot\! t^2 +\cdots +\!r \cdot t^{k-1} + \!t^{k})\!\!\!&\!\!=\!\!&\!\!a_{i,i+1}\cdot (\!r^{k+1}-\!t^{k+1})\!(r-t)^{-1}\\
\!\!\!&\!\!=\!\!&\!\!a_{i,i+1}\cdot(r-t)\!(r-t)^{-1}\\
\!\!\!&\!\!=\!\!&\!\!a_{i,i+1}
\end{array}
$$
and since $r^{k+1}=r$ and $t^{k+1}=t$ we concluded that $A(1,i)^{k+1}=A(1,i)$ regardless of $a_{i,i+1}$.
\item If $r=t$, then 
$$a_{i,i+1}\cdot(r^{k} + r^{k-1}\cdot t + r^{k-2}\cdot t^2 +\cdots +r \cdot t^{k-1} +t^{k})=a_{i,i+1}\cdot(k+1)\cdot r^{k}
$$
so we have the equation
$$a_{i,i+1}\cdot(k+1)\cdot r^{k}=a_{i,i+1},
$$
and as $char(\mathbb{K})\nmid k+1$, thus
\begin{itemize}
    \item  If $r\not=0$, then $a_{i,i+1}=0.$
    \item If $r=0$, then $a_{i,i+1}=0$.
\end{itemize}
\end{itemize}

Therefore, for $j-1=1$ we  have:
\begin{itemize}
    \item If $a_{ii} = a_{i+1,i+1}$, then $a_{i,i+1}=0$.
    \item If $a_{ii}\not= a_{i+1,i+1}$, then $a_{i,i+1}$ can be chosen arbitrarily.
\end{itemize}

Let $j-i=m$ be and assume now that for all $1\leq l\leq m-1$ we proved that if $A(l,i)^{k+1}=A(l,i)$ for all $i$, then $A(l,i)$ satisfy (ii) and (iii) of Theorem \ref{th1} and  we proved this to $m$. 

Suposse that $A(m,i)$ is as in the equation (\ref{eq11}) of the form $$A(m,i)=\left[ \begin{array}{ccc}
r& a & a_{i,i+m}\\ 
&u&  b \\
 &&t
\end{array}
\right]
$$
with $a_{ii}=r$, $a_{i+m,i+m}=t$ and $A(m,i)^{k+1}=A(m,i)$ where
$$A(m,i)^{k+1}=\left[ \begin{array}{ccc}
r^{k+1} & a' &x'\\ 
&u^{k+1}& b'\\
 &&t^{k+1}
\end{array}
\right].
$$

Analyzing $x'$ of  the equation (\ref{xx}) we have:

\begin{itemize}
    \item For $r\not=t$ we have of equation (\ref{eqk}) and  as $r^k=t^k=1$ then
$$\begin{array}{rcl}
x'&=& a_{i,i+m}\cdot \left(r^{k+1}-t^{k+1}\right)
\left(r-t\right)^{-1} \\
&&\\
&&+a\cdot [ \left(r^{k}-t^{k}\right)
\left(r-t\right)^{-1}\cdot e +\left(r^{k-1}-t^{k-1}\right)
\left(r-t\right)^{-1}\cdot u+\cdots \\
&&\\
&& +\left(r^{2}-t^{2}\right)
\left(r-t\right)^{-1}\cdot u^{k-2}+
\left(r-t\right)
\left(r-t\right)^{-1}\cdot u^{k-1} ]\cdot b,
\end{array}
$$
thus
\begin{equation}\label{eqmm}
\begin{array}{rl}
\!x'\!\!\!&\!=\!a_{i,i+m} \!+\!   a\!\cdot[\!r^{k}\cdot e +\!r^{k-1}\cdot u\!+\!\cdots +\!r^{2}\cdot u^{k-2}\!+\!r\cdot u^{k-1}\! +\!u^k-u^k ]\cdot
\left(r\!-\!t\right)^{-1}\cdot b\\
&\\
&- a\cdot[t^{k}\cdot e +t^{k-1}\cdot u+\cdots +t^{2}\cdot u^{k-2}+t\cdot u^{k-1}+u^k-u^k ]
\cdot\left(r-t)\right)^{-1}\cdot b.
\end{array}
\end{equation}
From equations (\ref{aa}) and (\ref{bb}) we have
$$
\begin{array}{rcl}
x'&=& a_{i,i+m} +\left[   (a'-a\cdot u^k)\cdot b-a\cdot(b'-b\cdot u^k )\right](r-t)^{-1},
\end{array}
$$
then
$$x'= a_{i,i+m}+ [a'\cdot b - a\cdot b'](r-t)^{-1}
$$
and since $A$ is $(k+1)$-potent $a=a'$ and $b=b'$, thus $a'\cdot b - a\cdot b'=0$, wich implies that $x'= a_{i,i+m}.$

Therefore, for  $a_{ii}\not=a_{i+m,i+m}$, $a_{i,i+m}$ can be chosen is arbitrarily.
\item For $r =t$, we have
$$
x'=a_{i,i+m}\cdot (k+1)\cdot r^k+a\cdot [k\cdot r^{k-1}\cdot e+\cdots +2  r \cdot u^{k-2}+u^{k-1}]\cdot b
$$
then
\begin{equation}\label{eqm}
x'=a_{i,i+m}\cdot k+a_{i,i+m}+ a\cdot [k\cdot r^{k-1}\cdot e+\cdots +2 r \cdot u^{k-2}+u^{k-1}]\cdot b
\end{equation}
so, since $char(\mathbb{K})\not= k+1$, the coefficient $a_{i,i+m}$ can be computed from equation $x'=a_{i,i+m}$, where
$$a_{i,i+m}\!=\!-k^{-1}\left\{\sum_{s=2}^{k+1}\!\sum_{i<i_1\leq i_2 \leq \cdots \leq i_{s-1}<i+m}(k\!+\!2\!-\!s)\cdot r^{k+1-s}a_{i,i_1}a_{i_1,i_2}\cdots a_{i_{s-1},i+m}\right\}.$$

\end{itemize}
Concluding, if $A(m,i)^{k+1}=A(m,i)$, then (ii) and (iii) hold.

Assume now that (i), (ii) and (iii) are satisfied for $A$. We shall prove that $A$ must be an $(k+1)$-potent matrix. Since the equation (\ref{eqth}) involves only the coefficients with indices $i_s$ such that $i\leq i_s\leq j$ it suffices to proved the claim for $A(m,i)$.

Since $a_{ii}\in\{0,1,\omega, \cdots, \omega^{k-1}\}$ and $\omega$ is a $kth$ root of unity we have that $a_{ii}^{k+1}=a_{ii}$.

Consider $m=1$. 
\begin{itemize}
    \item If  $a_{ii}\not=a_{i+1,i+1}$, then as it was checked in (\ref{eqm=1}) we have $A(1,i)^{k+1}=A(1,i)$.
    \item When $a_{ii}=a_{i+1,i+1}$ and $a_{i,i+1}$ satisfies (\ref{eqth}), we also have $A(1,i)^{k+1}=A(1,i)$.
\end{itemize}

Suposse that the claim holds for all $1\leq l\leq m-1$ and focus on $A(m,i)$. 
\begin{itemize}
\item If $a_{ii}\not=a_{i+m,i+m}$ and $a_{i,i+m}$ as in (\ref{eqth}). Consider 
$$A(m,i)=\left[ \begin{array}{ccc}
a_{ii}& a & a_{i,i+m}\\ 
&u&  b \\
 &&a_{i+m,i+m}
\end{array}
\right]
$$
as  in (\ref{eq11}), so from  (\ref{eq22}) we have that
$$A(m,i)^{k+1}=\left[ \begin{array}{ccc}
a_{ii}^{k+1} & a' &x'\\ 
&u^{k+1}& b'\\
 &&a_{i+m,i+m}^{k+1}
\end{array}
\right].
$$

By our assumption, $A(m-1,i)$ and $A(m-1,i+1)$ are $(k+1)$-potent for all $i$. So
$$A(m-1,i)=\left[\begin{array}{cc}
    a_{ii} &a  \\
    0 & u
\end{array}\right] \ \ \ and\ \ \  A(m-1,i+1)=\left[\begin{array}{cc}
    u &b  \\
    0 & a_{i+m,i+m}
\end{array}\right]$$
are $(k+1)$-potent. Hence
$$A(m-1,i)^{k+1}=\left[\begin{array}{cc}
    a_{ii}^{k+1} &a'  \\
    0 & u^{k+1}
\end{array}\right] \ \ \ and \ \ \ A(m-1,i+1)^{k+1}=\left[\begin{array}{cc}
    u^{k+1} &b'  \\
    0 & a_{i+m,i+m}^{k+1}
\end{array}\right],
$$
wich implies that $a=a',  b=b'$. Then, of the equation (\ref{eqmm}) we have $x'=a_{i,i+m}$.

\item Whereas, if $a_{ii}=a_{i+m,i+m}$ and $a_{i,i+m}$ satisfies (\ref{eqth}), then for the equation (\ref{eqm}), $x'=a_{i,i+m}$ as well. 
\end{itemize}
Summing up $A^{k+1}=A$

\end{proof}

\section{Counting triangular $(k+1)$-potent matrices}

In this section we used the Theorem \ref{th1} to count upper triangular $(k+1)$-potent matrices and prove our second result.


\begin{proof}[Proof of Theorem \ref{th2}] By Theorem \ref{th1}, the number of upper triangular $(k+1)$-potent matrices with elements of the set $D=\{0,1,\omega,\omega^2,\ldots,\omega^{k-1}\}$ on  the diagonal  depends entirely on which pairs of diagonal entries have $a_{ii}\not=a_{jj}$. To enumerate those possibilities, consider an integer column vector $d=(d_1,d_2,\ldots,d_n)$ the respective diagonal having each $d_i\in D$ and denote for $n_0,n_1,n_2,\cdots,n_k$ the numbers of $0,1,\omega,\omega^2,\cdots,\omega^{k-1}$ that appears in the diagonal respectively, such that $n_0+n_1+\cdots+n_k=n$ with $0\leq n_i$ for $i=0,1,\ldots,k$.

Denote by $\Delta$ the number de pairs $(d_i,d_j)$ with $i<j$ and $d_i\not= d_j$ and observe that
$$\Delta = \displaystyle \sum_{\substack{0\leq i<j \leq k }}n_i\cdot n_j.$$
In particular, $\Delta$ is independent of the order in which the elements of the set $D$ appear on $d$. Consequently we have on the diagonal  yields
$q^{\Delta} = q^{\displaystyle \sum_{\substack{0\leq i<j \leq k }}n_i\cdot n_j},$
possible upper triangular $(k+1)-$potent matrices.

Finally, all $d_i's$ can be put on our main diagonal on
$$\binom{n}{n_0}\binom{n-n_0}{n_1}\binom{n-n_0-n_1}{n_2}\cdots\binom{n-n_0-n_1-\cdots-n_{k-1}}{n_k}=\binom{n}{n_0,n_1,\ldots,n_k},$$
where
$$\binom{n}{n_0,n_1,\ldots,n_k}=\frac{n!}{n_0!\cdots n_k!}.
$$
Therefore, the total number of $n\times n$ upper triangular $(k+1)-$ potent matrices with elements the  set $\{0,1,\omega,\omega^2,\cdots,\omega^{k-1}\}$ on the diagonal is
$$ \displaystyle \sum_{\substack{n_0+\cdots+ n_k=n \\ 0\leq n_i}} \binom{n}{n_0n_2\cdots n_k}\cdot q^{\displaystyle \sum_{\substack{0\leq i<j \leq k }}n_i\cdot n_j}.
$$
\end{proof}

\section{Counting $(k+1)$-potent elements in $\mathcal{I}(X,\mathbb{K})$}
In this section, we fix the notation and recall some definitions and basic facts that will be used throughout the article. The notations and observations in this article are the same that in \cite{Ivan}.

We denoted by $X$ a nonempty finite partially ordered set (finite poset, for short) and its relation by $\leq$ and let $\mathbb{K}$ be a field. The incidence algebra $\mathcal{I}(X,\mathbb{K})$ of $X$ over $\mathbb{K}$ is the set $\mathcal{I}(X,\mathbb{K})=\{ f :X\times X\longrightarrow\mathbb{K}:\  f(x, y)=0 \ \text{if}\ x\not\leq y\}$, endowed with the usual product of a map by a scalar, the
usual sum of maps, and the product defined by
$$f\cdot g(x,y)=\sum_{x\leq t\leq y}f(x,t)g(t,y),$$
for any $ f, g \in \mathcal{I}(X,\mathbb{K})$.

Let $X=\{x_1,x_2,\ldots,x_n\}$ be a finite poset, the incidence algebra $\mathcal{I}(X,\mathbb{K})$ is a finite-dimensional linear space over $\mathbb{K}$, which has basis elements labelled $E_{ij}$ for each $i,j$ for wich $x_i\leq x_j$ and has multiplication defined via

$$
E_{ij}E_{kl}=\left\{\begin{array}{cc}
E_{ij}\ &,\  \text{if}\ j=k\\
0\ & ,\ \text{if}\ j\neq k
\end{array}
\right.
$$

It is not hard to see that we cant identify $\mathcal{I}(X,\mathbb{K})$ with a subalgebra of the full matrix algebra $M_n(\mathbb{K})$.

If $X=\{x_1, x_2,\ldots,x_n\}$ and the relation: $x_i\leq x_j$ if $i\leq j$. This   poset is called chain of length $n$. The Hasse diagram of this poset is

\begin{center}
\includegraphics[scale= 0.5]{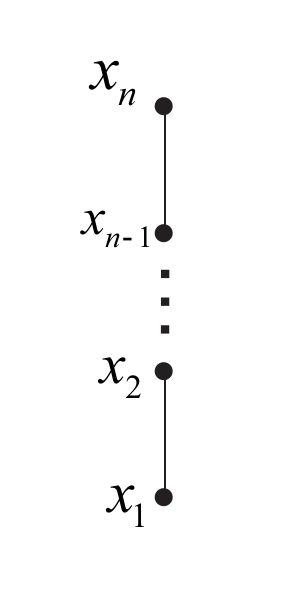}
\end{center}
and  $\mathcal{I}(X,\mathbb{K})\cong UT_n(\mathbb{K})$  the algebra of  upper triangular matrices.

Let $g\in \mathcal{I}(X,\mathbb{K})$, then $g$ is an $(k+1)$-potent element if $g^{k+1}=g$. For more information's about incidence algebras we recommended \cite{spi2}. 



\begin{proof}[Proof of Theorem \ref{teo1}]

The Hasse diagram of both posets are :

\begin{figure}[H]
\centering
{\resizebox*{10cm}{!}{\includegraphics{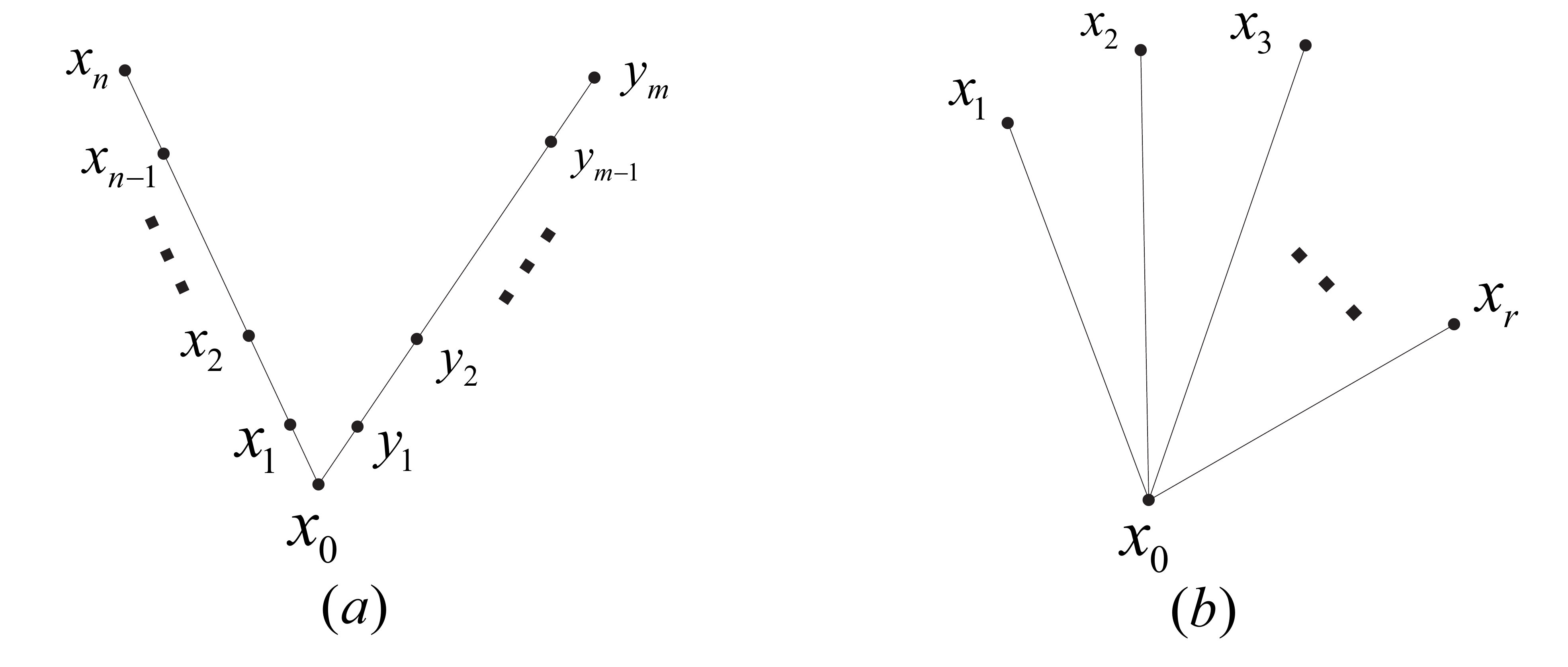}}}
 \label{graph4}
\end{figure}

Consider the poset of the figure ($a$) and suppose that there are $q_1,q_2,\cdots,q_s$ solutions in $\mathbb{K}$ of the equations $x^{k+1}=x$. The elements in this incidence algebra have the block matrix form

$$
g=\left[\begin{array}{c|c|c}
a_0 & \alpha & \theta \\ \hline
0 & R & 0 \\ \hline
0 & 0  & S

\end{array}\right]
$$
with $\alpha=(a_1,a_2,\cdots, a_n)$, $\theta=(\theta_1,\theta_2,\cdots, \theta_m)$, $R=(a_{ij})\in UT_n(\mathbb{K})$ and $S=(d_{ij})\in UT_m(\mathbb{K}).$

Suppose that $a_0=q_1$ and denote by
\begin{itemize}
    \item [(i)] $n_1,n_2,\cdots, n_s$ the number of $q_1,q_2,\cdots, q_s$ that appears in the diagonal of sub-matrix $R$, such that $n_1+n_2+\cdots+n_s=n$ and $0\leq n_i, \ i=1,2,\cdots,s$.
    
    \item [(ii)] $m_1,m_2,\cdots, m_s$ the number of $q_1,q_2,\cdots, q_s$ that appears in the diagonal of sub-matrix $S$, such that $m_1+m_2+\cdots+m_s=m$ and $0\leq m_i, \ i=1,2,\cdots,s$.
\end{itemize}
Then, in this case, using the matrix representation apply in \cite{Ivan}, we obtain that

 $$\begin{array}{ccccc}   & \overbrace{\begin{tabular}{|c|c|}  \hline $a_0$ & $(1, 0, \cdots, 0)$ \\ \hline
  $d_R$  & $n_1,n_2,\cdots,n_s$ \\ \hline  $d_S$ & $m_1,m_2,\cdots,m_s$ \\  \hline
\end{tabular}}^{Number   \ of \ q_1,q_2,\cdots, q_s  \ in:}  &  &  \left[\begin{array}{c|c|c}  0 & \displaystyle\sum_{i=2}^s n_i  & \displaystyle\sum_{i=2}^s m_i \\ \hline 0 & \Delta(d_R) & 0    \\ \hline 0 & 0 & \Delta(d_S) \end{array}\right]\end{array}$$
So, in this case we have

$$\binom{n}{n_1n_2\cdots n_s}\binom{m}{m_1m_2\cdots m_s} q^{\displaystyle \sum^s_{i=2}n_i+ \Delta(d_R)}\cdot q^{\Delta(d_S)}\cdot q^{\displaystyle \sum^s_{i=2}m_i}$$
but $\displaystyle \sum^s_{i=2}n_i+ \Delta(d_R)=\Delta(d_A)$ with $A=\left[\begin{array}{c|c} a_0 & \alpha \\ \hline 0 & R \end{array}\right]$
then, substituting in the formula above we have
$$\binom{n}{n_1n_2\cdots n_s}\binom{m}{m_1m_2\cdots m_s} q^{ \Delta(d_A)}\cdot q^{\Delta(d_S)}\cdot q^{\displaystyle \sum^s_{i=2}m_i}.$$
Combining this situation, we have
$$\displaystyle \sum_{\substack{n_1+\cdots +n_s=n \\ 0\leq n_i}}\sum_{\substack{m_1+\cdots +m_s=m \\ 0\leq m_i}}\binom{n}{n_1n_2\cdots n_s}\binom{m}{m_1m_2\cdots m_s} q^{ \Delta(d_A)}\cdot q^{\Delta(d_S)}\cdot q^{\displaystyle \sum^s_{i=2}m_i}= $$

$$\ \ \displaystyle \sum_{\substack{n_1+\cdots +n_s=n \\ 0\leq n_i}}\binom{n}{n_1n_2\cdots n_s}q^{ \Delta(d_A)} \left\{\sum_{\substack{m_1+\cdots +m_s=m \\ 0\leq m_i}}\binom{m}{m_1m_2\cdots m_s} q^{\Delta(d_S)}\cdot q^{\displaystyle \sum^s_{i=2}m_i}\right\}.$$
Observe that, by the Theorem \ref{th2}
$$\mathcal{N}(n+1,k+1,\mathbb{K})=\displaystyle \sum_{\substack{n_1+n_2+\cdots +n_s=n \\ 0\leq n_i}}\binom{n}{n_1n_2\cdots n_s}q^{ \Delta(d_A)}$$
and denote by 
$$P(m,k+1)=\sum_{\substack{m_1+m_2+\cdots +m_s=m \\ 0\leq m_i}}\binom{m}{m_1m_2\cdots m_s} q^{\Delta(d_S)}\cdot q^{\displaystyle \sum^s_{i=2}m_i},$$
then
$$ \mathcal{N}(X,k+1,\mathbb{K})=\mathcal{N}(n+1,k+1,\mathbb{K})\cdot P(m,k+1).$$
In general, consider the poset of the figure ($b$) then the proof follow similarly to \cite{Ivan} and we conclude that
$$\mathcal{N}(X,k+1,\mathbb{K})=\mathcal{N}(n+1,k+1,\mathbb{K})\cdot P(m_2,k+1)\cdots P(m_r,k+1).$$
\end{proof}


\begin{proof}[Proof of Theorem \ref{teo2}]
The Hasse diagram of Rhombuses poset is:

\begin{figure}[H]
\centering
{\resizebox*{6cm}{!}{\includegraphics{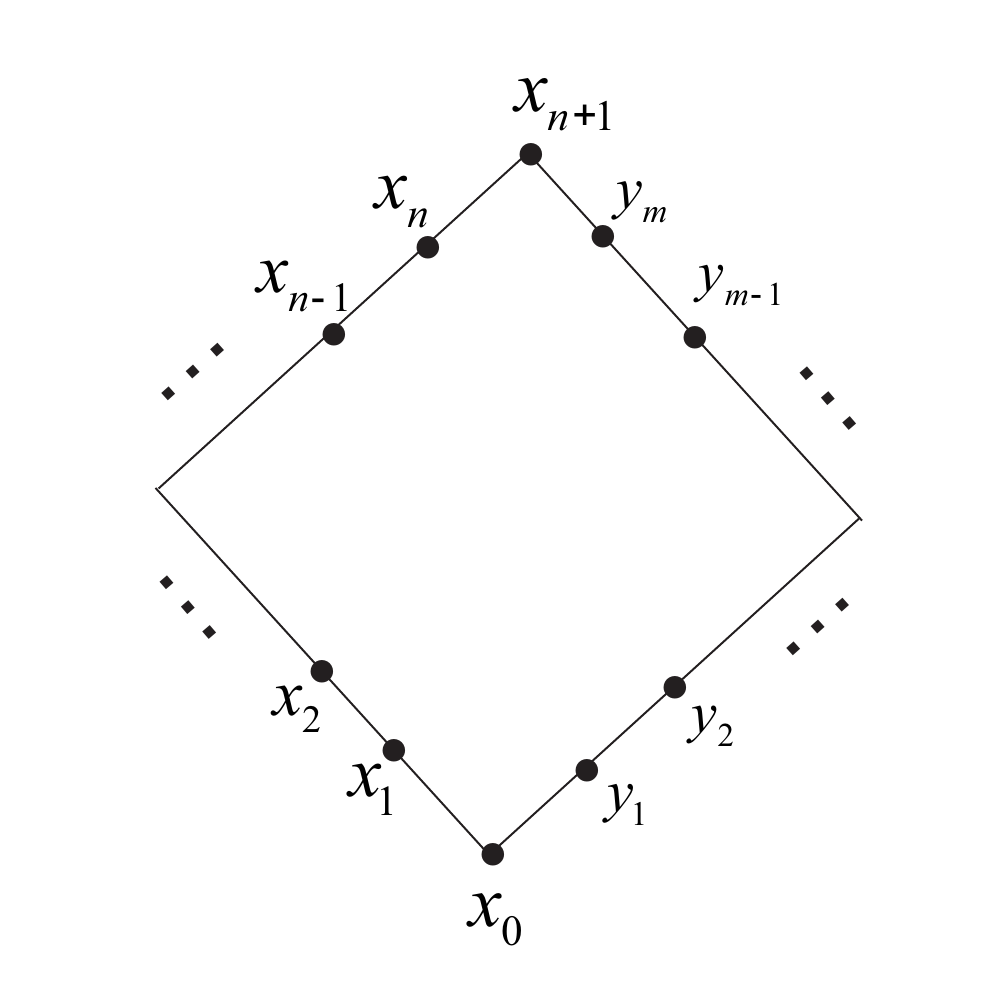}}}
\label{graph6}
\end{figure}

We know that the elements in this incidence algebra have the form:
$$\left[\begin{array}{c|c|c|c}
a_0 & \alpha & \theta & \theta_{n+1} \\ \hline
0   &  R     &   0    & \beta \\ \hline
0 &    0     &   S   & d \\ \hline
0   &  0    &    0   &  d_{(n+1)(n+1)}\end{array}\right]$$
where $\alpha=(\alpha_1,\alpha_2, \cdots ,\alpha_n)$, $\theta=(\theta_1 ,\theta_2 , \cdots  ,\theta_m)$ , $\beta=(\beta_1,\beta_2,\cdots , \beta_n)^t$, $R=(a_{ij})\in UT_n(\mathbb{K}) $ and $S=(d_{ij})\in UT_m(\mathbb{K})$. Consider that there are $q_1,q_2,\cdots, q_s$ different elements in $\mathbb{K}$ such that $x^{k+1}=x$. Suppose that $a_0=q_1$ and denote by 
\begin{itemize}
    \item [(i)] $n_1,n_2, \cdots, n_s$ the number of $q_1,q_2,\cdots, q_s$ that appears in the diagonal of the matrix $R$ respectively, such that $n_1+n_2+\cdots +n_s=n$, with $0\leq n_i$, $i=1,2,\cdots, s.$
    
    \item [(ii)] $m_1,m_2, \cdots, m_s$ the number of $q_1,q_2,\cdots, q_s$ that appears in the diagonal of the matrix $S$ respectively, such that $m_1+m_2+\cdots +m_s=m$, with $0\leq m_i$, $i=1,2,\cdots, s.$
    \end{itemize}
    
    Denote by
    $$\Delta(d_R)= \displaystyle \sum_{\substack{0\leq i,j \leq s \\ i<j}}n_i\cdot n_j\ \ and \ \ \Delta(d_S)= \displaystyle \sum_{\substack{0\leq i,j \leq s \\ i<j}}m_i\cdot m_j,$$
    then, consider the $s$ choices for $d_{(n+1)(n+1)}$ and using the matrix representation explain in \cite{Ivan} we obtain 
    
    \begin{itemize}
        \item [(a)] If $d_{(n+1)(n+1)}=q_1$ then

      $$\begin{array}{ccccc}   & \overbrace{\begin{tabular}{|c|c|}\hline
  $d_R$  & $n_1,n_2,\cdots,n_s$ \\ \hline  $d_S$ & $m_1,m_2,\cdots,m_s$ \\ \hline
   $d_{(n+1)(n+1)}$  & $1,0,\cdots, 0$ \\ \hline
\end{tabular}}^{Number   \ of \ q_1,q_2,\cdots, q_s \ in:}  &  &  \left[\begin{array}{c|c|c|c}  0 & \displaystyle\sum^s_{i=2}n_i & \displaystyle\sum^s_{i=2}m_i & 0 \\ \hline 0 & \Delta(d_R) & 0 & \displaystyle\sum^s_{i=2}n_i    \\ \hline 0 & 0 & \Delta(d_S) &  \displaystyle\sum^s_{i=2}m_i\\  \hline 0 & 0 & 0 & 0 \end{array}\right]\end{array}$$ 
So, in this case we have
$$\binom{n}{n_1n_2\cdots n_s}\binom{m}{m_1m_2\cdots m_s}q^{\Delta(d_R)+\Delta(d_S)}\cdot q^{2(\displaystyle\displaystyle\sum^s_{i=2}n_i+\sum^s_{i=2}m_i)}.$$

    \item[(b)] If $d_{(n+1)(n+1)}=q_j$ with $j=2,3,\cdots, s$ then

 $$\begin{array}{ccccc}   & \overbrace{\begin{tabular}{|c|c|}\hline
  $d_R$  & $n_1,n_2,\cdots,n_s$ \\ \hline  $d_S$ & $m_1,m_2,\cdots,m_s$ \\ \hline
   $d_{(n+1)(n+1)}$  & $1,0,\cdots, 0$ \\ \hline
\end{tabular}}^{Number   \ of \ q_1,q_2,\cdots, q_s \ in:}  &  &  \left[\begin{array}{c|c|c|c}  0 & \displaystyle\sum^s_{i=2 }n_i & \displaystyle\sum^s_{i=2}m_i & 1 \\ \hline 0 & \Delta(d_R) & 0 & \displaystyle\sum^s_{i=1,i\neq j}n_i    \\ \hline 0 & 0 & \Delta(d_S) &  \displaystyle\sum^s_{i=1, i\neq j}m_i\\  \hline 0 & 0 & 0 & 0 \end{array}\right]\end{array}$$ 

So, in this case we have
$$\binom{n}{n_1n_2\cdots n_s}\binom{m}{m_1m_2\cdots m_s}q^{\Delta(d_R)+\Delta(d_S)}\cdot q\cdot q^{\displaystyle\sum^s_{i=2 }(n_i+m_i)+ \sum^s_{i=1,i\neq j}(n_i +m_i)}.$$
    \end{itemize}
Finally, combining all situations, we have

$$s\cdot \displaystyle \sum_{\substack{n_1+\cdots+ n_s=n \\ 0\leq n_i}} \sum_{\substack{m_1+\cdots +m_s=m \\ 0\leq m_i}} \binom{n}{n_1n_2\cdots n_s}\binom{m}{m_1m_2\cdots m_s}\cdot q^{\Delta(d_R)+\Delta(d_D)}\times$$ $$\times q^{\displaystyle\sum^s_{i=2 }(n_i+m_i)}\times\left\{q^{\displaystyle\sum^s_{i=2 }(n_i+m_i)}+ q\cdot \sum_{j=2}^s q^{\displaystyle\sum^s_{i=1,i\neq j}(n_i +m_i)}\right\},$$
$(k+1)$-potent elements in $\mathcal{I}(X, \mathbb{K}).$

\end{proof}


\begin{proof}[Proof of Theorem \ref{teo3}]
The Hasse diagram for this poset is
\begin{figure}[H]
\centering
{\resizebox*{5,5cm}{!}{\includegraphics{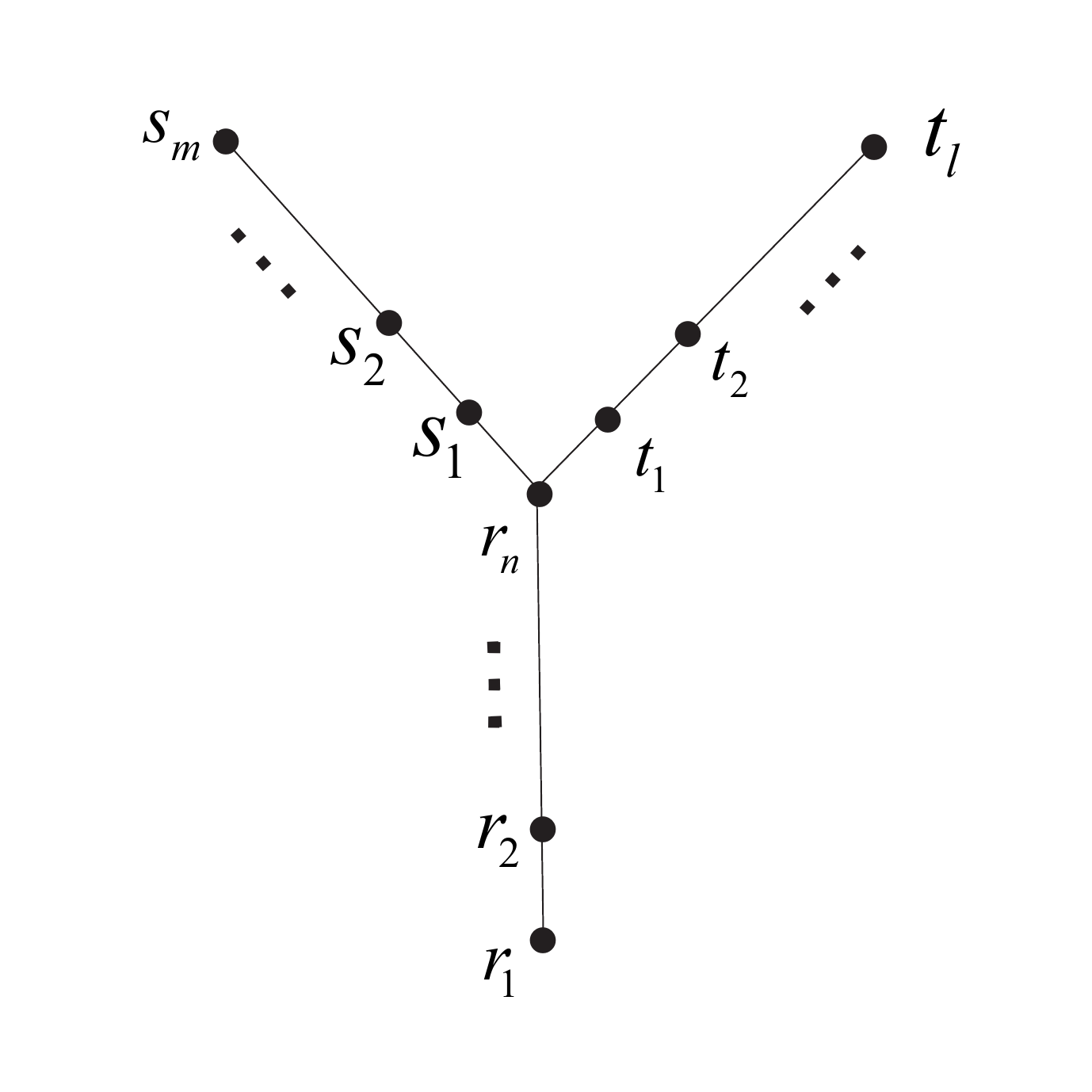}}}
\label{graph7}
\end{figure}
The elements in the incidence algebra $\mathcal{I}(Y,\mathbb{K})$ are in the matrix form:

$$\left[\begin{array}{c|c|c} R & *_1 & *_2 \\ \hline 0 & S & 0 \\ \hline 0 & 0 & T\end{array}\right],$$ 
with $R\in UT_n(\mathbb{K})$, $S\in UT_m(\mathbb{K})$, $T\in UT_l(\mathbb{K})$, $*_1 \in M_{n\times m}(\mathbb{K})$ and $*_2\in M_{n \times l}(\mathbb{K})$. 

In geral, suppose that we have $q_1,q_2,\cdots ,q_s$ different elements in $\mathbb{K}$ such that $x^{k+1}=x$ and denote by
\begin{itemize}
    \item [(i)] $n_1,n_2,\cdots, n_s$ the number of $q_1,q_2,\cdots,q_s $ that appears in the diagonal $d_R$ such that $n_1+n_2+\cdots +n_s=n$, where $0\leq n_i$, $i=1,2,\cdots, s$.
     \item [(i)] $m_1,m_2,\cdots, m_s$ the number of $q_1,q_2,\cdots,q_s $ that appears in the diagonal $d_S$ such that $m_1+m_2+\cdots +m_s=m$, where $0\leq m_i$, $i=1,2,\cdots, s$.
      \item [(i)] $l_1,l_2,\cdots, l_s$ the number of $q_1,q_2,\cdots,q_s $ that appears in the diagonal $d_T$ such that $l_1+l_2+\cdots +l_s=l$, where $0\leq l_i$, $i=1,2,\cdots, s$.
    \end{itemize}

Then, using the matrix representation explain in \cite{Ivan} we obtain that

$$\begin{array}{ccccc}   & \overbrace{\begin{tabular}{|c|c|} \hline $d_R$ & $(n_1,n_2,\cdots,n_s)$ \\ \hline $d_S$ & $(m_1,m_2,\cdots,m_s)$ \\  \hline $d_T$ & $(l_1,l_2,\cdots,l_s)$ \\ \hline \end{tabular}}^{Number   \ of \ q_1,q_2 \ and \ q_3 \ in:}  &  &  \left[\begin{array}{c|c|c} \Delta(d_R) &  \displaystyle\sum^{s}_{\substack{0\leq,i,j\leq s \\ i \neq j}} n_i\cdot m_j & \displaystyle\sum^{s}_{\substack{0\leq,i,j\leq s \\ i \neq j}} n_i\cdot l_j  \\ \hline 0 & \Delta(d_S) & 0 \\ \hline 0 & 0 & \Delta(d_T)\end{array}\right]\end{array},$$
where
$$\Delta(d_R)= \displaystyle\sum^{s}_{\substack{0\leq,i,j\leq s \\ i < j}} n_i\cdot n_j, \ \Delta(d_S)= \displaystyle\sum^{s}_{\substack{0\leq,i,j\leq s \\ i < j}} m_i\cdot m_j \ and \  \Delta(d_T)= \displaystyle\sum^{T}_{\substack{0\leq,i,j\leq s \\ i < j}} l_i\cdot l_j$$
    
Combining all situations, we  have

$$\displaystyle \sum_{\substack{n_1+n_2+\cdots +n_s=n  \\ m_1+m_2+\cdots+m_s=m \\ l_1+l_2+\cdots +l_s=l \\ 0\leq n_i,m_i,l_i, \ i=1,2,\cdots, s}} \binom{n}{n_1n_2\cdots n_s}\binom{m}{m_1m_2\cdots m_s} \binom{l}{l_1l_2\cdots l_s}\cdot q^{\Delta(d_R)+\Delta(d_S)+\Delta(d_T)}\times$$
$$\times\ q^{\displaystyle\sum^{s}_{\substack{0\leq,i,j\leq s \\ i \neq j}} n_i\cdot m_j+\displaystyle\sum^{s}_{\substack{0\leq,i,j\leq s \\ i \neq j}} n_i\cdot l_j },$$
$(k+1)$-potent elements in $\mathcal{I}(X, \mathbb{K}).$


\end{proof}

\section{Remark}
Slowik in \cite{Slowik} proof the similar result to Theorem \ref{th2} but for triangular matrices of finite order:

\begin{theorem}\label{sw2} Let $h \in \mathbb{N}$ and let $\mathbb{K}$ be a finite field such that
$char\mathbb{K}\nmid h$ and $|\mathbb{K}| = q$. If $\mathbb{K}$ contains $l$ elements $d$ such that $d^h = 1$, then $T_n(\mathbb{K})$ contains
\begin{equation}\label{equaslowik}
\!\!\mathcal{N}(n,h,\mathbb{K})\!=\!\!\!\!\displaystyle\sum_{j=1}^{\min(l,n)}\!\!\!\! \displaystyle\sum_{\substack{m_1\leq m_2\leq\!\cdots\!\!\leq m_j \\ m_1+\!\cdots\!+ m_j=n}}\!\!\!\!\frac{l!}{(l\!-\!j)! \!\cdot\! g(m_1,\cdots,m_j)}\binom{n}{m_1,\!\!\cdots\!\!,m_j}q^{\frac{1}{2}(n^2\!-\!\sum_{u=1}^j m^2_u)}, 
\end{equation}
elements $g$ satisfying $g^h = e$, where $e$ is the identity matrix in $T_n(\mathbb{K})$ and $g(m_1,\!\cdots\!, m_ j)\! =\! r_1!\cdots r_s!$ if and only if
$$m_1 = m_2 = \cdots = m_{r_1} \neq m_{r_1+1} = m_{m_1+2} = \cdots = m_{r_1+r_2} \neq m_{r_1+r_2+1} =\cdots.$$
\end{theorem}
It is not difficult to observe that  the  formula in (\ref{equaslowik}) can be rewritten as
\begin{equation}\label{eq2}
\mathcal{N}(n,h,\mathbb{K})=\displaystyle\sum_{\substack{n_1+n_2+\cdots+n_l=n \\ 0\leq n_i, \ i=1,2,\cdots,l}}\binom{n}{n_1n_2\cdots n_l}q^{\displaystyle\sum_{1\leq i<j\leq l}n_in_j}.
\end{equation}

For example, for $j=1$, the formula (\ref{equaslowik}) corresponding to count all combinations of $ n_i=n$ and $n_j=0$, with $i\neq j$. If consider $n_1=n$ and $n_2=\cdots =n_l=0$ then we have $\binom{n}{n}q^0=1$. Combining all situations, in total we have $1+1+\cdots+1=l$ of this elements.
For $j=2$ in the Slowik formula for the case that $m_1+m_2=n$ corresponding to count all combinations of $n_i=m_1, n_j=m_2$ and $n_k=0$ with $i\neq j$, $j\neq k$ and $i\neq k$. In this case we have $\binom{n}{(n-2),g}$ where $g=2$ if $n_1=n_2$ and $g=1$ in other cases.
In general, for $j<n$ fixed in the Slowik's formula we have $m_1+m_2+m_3+\cdots+m_j=n$ that corresponding to count all combinations for $n_{i_1}=m_1, n_{i_2}=m_2, \cdots, n_{i_j}=m_j$ and $n_{i_k}=0$ in other cases, with $i_s \in \{1,2,\cdots, n\}$ and $i_s\neq i_t$ if $s\neq t$, here we have $\binom{n}{(n-j)g}$ for this combinations, where $g$ is the variable that appears in the Theorem \ref{sw2}.

In conclusion, the Slowik's formula that appears in the Theorem \ref{sw2} is equals to formula of the Theorem \ref{th2}, both formula are compute for different classes of matrices. The authors in a future article will explore these remark, for instance, the authors proofs an similar remark but for full matrix groups (to appears).
\newline



\section{Tables presentations}
Here we present some examples of triangular matrices satisfying the Theorem  \ref{th1} where $a,b,c, d, e, f ,g,h, l$ are arbitrary and $X$ is the dependent variable.
\begin{table}[H]
\caption{Case $k=4$, $n=3$ and $1+\omega+\omega^2+\omega^3=0$ }
\centering
\begin{tabular}{|c||c||c|}\hline
$
\begin{array}{c}
\left[\begin{array}{ccc}
  1&a&b  \\
0&\omega^2&b  \\
  0&0&\omega^3
 \end{array}
\right]
\end{array}
$& 
$\begin{array}{c}
\left[\begin{array}{ccc}
  \omega^3&a&X  \\
0&\omega^2&b  \\
  0&0&\omega^3
 \end{array}
\right]\\
\\
 X= -\frac{1}{4} (2+2\omega)ab
 \end{array}$
 &
$\begin{array}{c}
\left[\begin{array}{ccc}
  \omega&a&X  \\
0&\omega^3&b  \\
  0&0&\omega
 \end{array}
\right]\\
\\
X=  -\frac{1}{4} (2\omega^3)ab
 \end{array}$ \\ \hline
&&\\
$\begin{array}{c}
\left[\begin{array}{ccc}
  1&a&X  \\
0&\omega^3&b  \\
  0&0&1
 \end{array}
\right]\\
\\
X= -\frac{1}{4} (2+2\omega^3)ab
\end{array}$&
$\begin{array}{c}
\left[\begin{array}{ccc}
  0&a&X  \\
0&\omega^2&b  \\
  0&0&0
 \end{array}
\right]\\
\\
X= -\frac{1}{4}(-4\omega^2)ab
\end{array}$&

$\begin{array}{c}
\left[\begin{array}{ccc}
  1&a&X  \\
0&\omega^2&b  \\
  0&0&1
 \end{array}
\right]\\
\\
X=  -\frac{1}{4} (2)ab
\end{array}$
\\  \hline
&&\\
$
\begin{array}{c}
\left[\begin{array}{ccc}
  \omega^2&a&X  \\
0&0&b  \\
  0&0&\omega^2
 \end{array}
\right]\\
\\
X= -\frac{1}{4}(4\omega^2)ab
\end{array}$&$
\begin{array}{c}
\left[\begin{array}{ccc}
  \omega&a&X  \\
0&\omega^3&b  \\
  0&0&\omega
 \end{array}
\right]\\
\\
X= -\frac{1}{4} (2\omega^3)ab
\end{array}
$&$
\begin{array}{c}
\left[\begin{array}{ccc}
  1&a&X  \\
0&0&b  \\
  0&0&1
 \end{array}
\right]\\
\\
X=  -\frac{1}{4} (4)ab
\end{array}
$ \\
\hline
\end{tabular}
\end{table}

\begin{table}[H]
\caption{Case $k=3$, $n=4$ and $1+\omega+\omega^2=0$}
\centering
\begin{tabular}{|c||c|}\hline
&\\
$\begin{array}{c}
\left[\begin{array}{cccc}
  \omega&0&a&b   \\
0&\omega&c&d   \\
  0&0&1& 0 \\
0  &0&0&1
 \end{array}
\right]
\end{array}$& $\begin{array}{c}
\left[\begin{array}{cccc}
  1&a&d&X   \\
0&\omega&b&e   \\
  0&0&\omega^2& c \\
0  &0&0&1
 \end{array}
\right]\\
\\
\begin{array}{rcl}
  \!\!\!  X\!\!\! &\!\!=\!\!&\!\!-\frac{1}{3}\left[(2+\omega)ae+(2+\omega^2)dc+\ abc\right]
\end{array}
\end{array}$\\  \hline
&\\
$\begin{array}{c}
\left[\begin{array}{cccc}
 \omega&a&d&e   \\
0&\omega^2&b&X   \\
0&0&1&c \\
0 &0&0&\omega^2  
 \end{array}
\right]\\
\\
X= -\frac{1}{3}(2\omega+\omega^2)bc
\end{array}$&
$\begin{array}{c}
\left[\begin{array}{cccc}
 \omega&a&d&X   \\
0&\omega^2&b&e   \\
0&0&1& c \\
0  &0&0&\omega  
 \end{array}
\right]\\
\\
\begin{array}{rcl}
 \!\! \!\!\!  X &\!\!=\!\!&\!\!-\frac{1}{3}\left[(1+2\omega^2)ae+(\omega+2\omega^2)dc+ \omega\ abc\right] 
\end{array}
\end{array}$
\\  \hline
&\\

$\begin{array}{c}
\left[\begin{array}{cccc}
 \omega&a&X&c   \\
0&1&b&d   \\
 0&0&\omega& e \\
0  &0&0&\omega^2  
 \end{array}
\right]\\
\\
X=  -\frac{1}{3}(\omega+2\omega^2)ab
 \end{array}$
&
$\begin{array}{c}
\left[\begin{array}{cccc}
  \omega&a&d&X   \\
0&\omega^2&0&e   \\
  0&0&\omega^2& c \\
0  &0&0&\omega
 \end{array}
\right]\\
\\
X= -\frac{1}{3}\left[(1+2\omega^2)ae+(1+2\omega^2)dc\right]
\end{array}$\\ \hline
\end{tabular}
\end{table}

\begin{table}[H]
\caption{Case $k=4$, $n=5$ and $1+\omega+\omega^2+\omega^3=0$}
\centering
\begin{tabular}{|c||c|}\hline
&\\
$
\left[\begin{array}{ccccc}
  \omega&a&b&c &j  \\
0&\omega^3&d&e &f  \\
  0&0&\omega^2& g& h \\
0  &0&0&1 &l  \\
 0 &0&0&0 &0  
 \end{array}
\right]
$& 
$\begin{array}{c}
\left[\begin{array}{ccccc}
  \omega^3&a&X&b &c  \\
0&\omega^2&d&e &f  \\
 0&0&\omega^3& g& h \\
0  &0&0&1 &l  \\
 0 &0&0&0 &\omega  
 \end{array}
\right] \\
 \\
X=  -\frac{1}{4}(2+2\omega)ad
 \end{array}$\\  \hline
&\\
$\begin{array}{c}
\left[\begin{array}{ccccc}
 \omega^3&a&b&c &d  \\
0&\omega^3&e&X &f  \\
 0&0&\omega^2& g&h \\
0  &0&0&\omega^3 &l  \\
 0 &0&0&0 &1  
 \end{array}
\right]\\
\\
 X=  -\frac{1}{4}(2\omega^2+2\omega^3)hl
\end{array}$
&
$\begin{array}{c}
\left[\begin{array}{ccccc}
  \omega^3&a&b&c &d  \\
0&\omega^3&e&X &f  \\
  0&0&\omega^2& g& h \\
0  &0&0&\omega^3 &l  \\
 0 &0&0&0 &1  
 \end{array}
\right]\\
\\
X= -\frac{1}{4}(2+2\omega)eg
\end{array}$
\\  \hline
&\\
$\begin{array}{c}
\left[\begin{array}{ccccc}
  0&a&b&X &c  \\
0&\omega^3&d&e &f  \\
0&0&\omega& g&h \\
0  &0&0&0 &l  \\
 0 &0&0&0 &1  
 \end{array}
\right]\\
\\
X= \omega\ ae +\omega^3\ bg+ 
 \omega^2\ adg
\end{array}$& 
$\begin{array}{c}
\left[\begin{array}{ccccc}
  \omega&a&b&c &X  \\
0&\omega^3&d&e &f  \\
 0&0&\omega^2& g&h \\
0  &0&0&1 &l  \\
 0 &0&0&0 &\omega  
 \end{array}
\right]\\
\\
\begin{array}{rcl}
\!\!X&\!=\!& -\frac{1}{4}\left[ 2\omega^3\ af +(2+2\omega^3)bh++(2\omega^2+2\omega^3)cl\right.\\
&&\\
&&+\left.(\omega^2+\omega^3)adh+(\omega+\omega^2)ael+2\omega^2\ bgl\right.\\
&&\\
&&\left.+\omega\ adgi\right]
\end{array} 
\end{array}
$\\ \hline
\end{tabular}
\end{table}
In the Tables 4 to 8 we present the values of $P(m,k+1)$ for $1\leq m \leq 7$ and $2\leq k \leq 5$. In the Tables 9, 10 and 11, we present the number$(k+1)$-potent elements on the Rhombuses poset for some   values $m,n$ and $2\leq k \leq 5 $ and finally, in the Table 8 present the number $(k+1)$-potent elements in the $Y$ poset for some values of $n,m,l$ and $k=2$.

\begin{table}[H]
\caption{Numbers of $3$-potent elements of the  Star  Poset}
\centering
\begin{tabular}{|c||cl|}\hline
$m$ &  &  \multicolumn{1}{c}{$P(m,3)$}\\ [0.3ex]
\hline\hline
3 & & $12q^5+6q^4+8q^3+1$\\ 
4 &  &  $30q^8+20q^7+20q^6+10q^4+1$\\ 
5 &  & $30q^{12}+120q^{11}+50q^9+30q^8+12q^5+1 $\\ 
6 &  &$210q^{16}+140q^{15}+210q^{14}+70q^{12}+42q^{11}+42q^{10}+14q^6+1 $\\ 
7 &  & $560q^{21}+420q^{20}+560q^{19}+340q^{17}+70q^{16}+112q^{15}+56q^{13}+$  \\
 & & $56q^{12}+16q^{7}+1 $\\ 
 [1ex]
\hline
\end{tabular}
\label{}
\end{table}

\begin{table}[H]
\caption{Numbers $4$-potent elements  of the  Star  Poset}
\centering
\begin{tabular}{|c||cl|}\hline
$m$ &  &  \multicolumn{1}{c}{$P(m,4)$}\\ [0.4ex]
\hline\hline
4 & & $60q^9+90q^8+60q^7+30q^6+15q^4+1$\\ 
5 &  &  $270q^{13}+210x^{12}+360q^{11}+120q^9+45q^8+18q^5+1$\\ 
6 &  & $630q^{18}+1260q^{17}+630q^{16}+630q^{15}+630q^{14}+105q^{12}+126q^{11}+63q^{10}+$\\ 
 &  & $+21q^{6}+1$\\
7 &  & $630q^{24}+5040q^{23}+1680q^{22}+4200q^{21}+1260q^{20}+1680q^{19}+336q^{18}+$\\
& & $+1008q^{17}+105q^{16}+168q^{15}+168q^{13}+84q^{12}+24q^7+1 $\\
 [1ex]
\hline
\end{tabular}
\label{}
\end{table}

\begin{table}[H]
\caption{Numbers of $5$-potent elements of the  Star  Poset}
\centering
\begin{tabular}{|c||cl|}\hline
$m$ &  &  \multicolumn{1}{c}{$P(m,5)$}\\ [0.6ex]
\hline\hline
4 & & $24q^{10}+240q^9+180q^8+120q^7+40q^6+20q^4+1$\\ 
5 &  &  $360q^{14}+1080q^{13}+660q^{12}+720q^{11}+220q^9+60q^8+24q^5+1$\\ 
6 &  & $ 2520q^{19}+3360q^{18}+5040q^{17}+1260q^{16}+1680q^{15}+1260q^{14}+140q^{12}+$\\
  & & $+252q^{11}+84q^{10}+28q^6+1$\\ 
7 &  & $10080q^{25}\!+\! 15960q^{24}\!+\!20160q^{23}\!+\!8400x^{22}\! +\!134400q^{21}\!+\!2520q^{20}\!+\!3360q^{19}\!+\!$\\
 & & $+1344q^{18}+2016q^{17}+140q^{16}+224q^{15}+336q^{13}+112q^{12}+32q^{7}+1$\\ 
 [1ex]
\hline
\end{tabular}
\label{}
\end{table}

\begin{table}[H]
\caption{Numbers of $6$-potent elements  of the  Star  Poset}
\centering
\begin{tabular}{|c||cl|}\hline
$m$ &  & \multicolumn{1}{c}{$P(m,6)$}\\ [0.6ex]
\hline\hline
4 & & $120q^{10}+600q^{9}+300q^8+200q^7+50q^6+25q^4+1.$\\ 
5 &  &  $120q^{15}\!+\!1800q^{14}\!+\!2700q^{13}\!+\!1500q^{12}+1200q^{11}+350q^9+75q^8+30q^5\!+\!1.$\\
6 &  & $2520q^{20}+126000q^{19}+10500q^{18}+126000q^{17}+2100q^{16}+3500q^{15}$ \\   &  & $+2100q^{14}+175q^{12}+420q^{11}+105q^{10}+35q^6+1$. \\ 
7 &  & $25200q^{26}\!\!+\!57120q^{25}\!+\!73500q^{24}\!+\!50400q^{23}\!+\!25200q^{22}\!+\!30800q^{21}\!+\!4200q^{20}\!+\!$\\
& &  $+5600q^{19}\!+\!3360q^{18}\!+\!3360q^{17}\!+\!175q^{16}\!+\!280q^{15}\!+\!560q^{13}\!+\!140q^{12}\!+\!40q^7\!+\!1$\\[1ex]
\hline
\end{tabular}
\label{}
\end{table}


\begin{table}[H]
\caption{Numbers of $3$-potent elements in the Rhombuses Poset}
\centering
\begin{tabular}{|c|c||l|}\hline
$n$ & $m$   &  \multicolumn{1}{c}{  $k+1=3$}\\ [0.3ex]
\hline\hline
2 & 2  & $12q^{10}+174q^{9}+180q^8+192q^7+
96q^6+36q^5+12q^4+24q^3+3.$ \\
2 & 3   & $36q^{13}\!+\!432q^{12}\!+\!510q^{11}\!+\!528q^{10}\!+\!288q^9\!+\!186q^8\!+\!114q^7\!+\!48q^6\!+\!6q^5\!+\!$ \\ 
 &   &    $+24q^4+12q^3+3.$ \\
2 & 4 &   $36q^{17}\!+\!588q^{16}\!+\!1728q^{15}\!+\!1176q^{14}\!+\!1182q^{13}\!+\!762q^{12}\!+\!540q^{11}\!+\!144q^{10}\!+\!$ \\
 &    &   $+168q^9+162q^8+24q^7+30q^5+6q^4+12q^3+3$ \\   
2 & 5 &     $300q^{21}+2940q^{20}+3720q^{19}+4560q^{18}+2460q^{17}+1890q^{16}+1302q^{15}+$\\  
&  &   $+1254q^{14}+516q^{13}+210q^{12}+120q^{11}+186q^{10}+144q^9+24q^8+30q^6+$\\
 & &   $6q^5+6q^4+12q^3+3$\\
3 & 2 &  $36q^{13}\!+\!432q^{12}\!+\!510q^{11}\!+\!528q^{10}\!+\!288q^9\!+\!186q^8\!+\!114q^7\!+\!48q^6\!+\!6q^5\!+$\\
 &  &   $\!24q^4\!+12q^3+3$\\
3 & 3 &   $108q^{16}+1080q^{15}+1404q^{14}+1446q^{13}+912q^{12}+720q^{11}+456q^{10}+$ \\
 &  &    $144q^9+144q^8+60q^7+48q^6+36q^4+3$\\ 
 3& 4 &   $108q^{20}+1548q^{19}+4392q^{18}+3408q^{17}+3384q^{16}+2658q^{15}+1740q^{14}+$ \\
  &  &    $+792q^{13}+738q^{12}+468q^{11}+72q^{10}+204q^9+84q^8+18q^7+24q^6+$  \\
  & &  $24q^5+18q^4+3$\\
 3 & 5 &  $900q^{24}+7380q^{23}+10140q^{22}+12240q^{21}+7860q^{20}+6540q^{19}+$ \\
  &  &   $4968q^{18}+3618q^{17}+2034q^{16}+1080q^{15}+918q^{14}+444q^{13}+414q^{12}+$ \\ 
  & &  $132q^{11}+246q^{10}+24q^9+18q^8+18q^7+54q^6+18q^4+3$\\ [1ex]
\hline
\end{tabular}
\label{}
\end{table}

\begin{table}[H]
\caption{Numbers of $4$-potent elements in  the Rhombuses Poset}
\centering
\begin{tabular}{|c|c||l|}\hline
$n$ & $m$ &    \multicolumn{1}{c}{  $k+1=4$}\\ [0.3ex]
\hline\hline
3 & 2 &  $576q^{14}+3240q^{13}+4440q^{12}+3768q^{11}+2208q^{10}+1056q^9+564q^8+$ \\
& &  $336q^7+96q^6+24q^5+48q^4+24q^3+4$ \\
3 & 3 & $144q^{18}+4752q^{17}+12528q^{16}+16128q^{15}+13680q^{14}+7968q^{13}+$ \\
& &  $4752q^{12}\!+\!2880q^{11}\!+\!1536q^{10}\!+\!336q^9\!+\!468q^8\!+\!192q^7v+\!96q^6+72q^4+4$ \\
3& 4 &  $864q^{22}\!+\!20448q^{21}\!+\!42480q^{20}\!+\!59904q^{19}\!+\!50448q^{18}\!+\!33864q^{17}\!+ $\\ 
& &  $\!23328q^{16}\!+\!14952q^{15}\!+\!6480q^{14}\!+\!3744q^{13}\!+\!2940q^{12}\!+\!1488q^{11}\!+\!144q^{10}\!+\!$\\
& &   $648q^9+204q^8+72q^7+48q^6+48q^5+36q^4+4$\\
3 & 5 &  $2160q^{27}\!+\!51120*q^{26}\!+\!152640q^{25}+199440q^{24}\!+\!188880q^{23}\!+\!163320q^{22}\!+\!$\\
  &  &    $+114840q^{21}+66960q^{20}+39264q^{19}+32232q^{18}+17592q^{17}+7068q^{16}+$\\
  & &   $4440q^{15}+4068q^{14}+1896q^{13}+1236q^{12}+384q^{11}+672q^{10}+72q^9+$ \\
  & &  $72q^8+72q^7+108q^6+36q^4+4$\\ [1ex]
\hline
\end{tabular}
\label{}
\end{table}

\begin{table}[H]
\caption{Numbers of $5$-potent elements in the Rhombuses Poset}
\centering
\begin{tabular}{|c|c||l|}\hline
$n$ & $m$ &    \multicolumn{1}{c}{  $k+1=5$}\\ [0.3ex]
\hline\hline
3& 4 &  $12960q^{23}\!+\!140400q^{22}\!+\!375120q^{21}\!+\!436800q^{20}\!+\!415800q^{19}\!+\!239640q^{18}\!+\!$ \\
& &  $+154440q^{17}+86400q^{16}+48500q^{15}+17120q^{14}+12000q^{13}+7980q^{12}+$\\  
& &  $3400q^{11}+240q^{10}+1520q^9+400q^8+180q^7+80q^6+80q^5+60q^4+5$ \\
3 & 5 &  $97200q^{28}+802800q^{27}+1555200q^{26}+2202600q^{25}+1723200q^{24}+ $\\
& &  $1318800q^{23}+903600q^{22}+516000q^{21}+267360q^{20}+153720q^{19}+ $ \\
& &  $116880q^{18}+54180q^{17}+16780q^{16}+14240q^{15}+11820q^{14}+5440q^{13}+$ \\ 
& &  $2740q^{12}+840q^{11}+1420q^{10}+200q^9+180q^8+180q^7+180q^6+60q^4+5$\\ [1ex]
\hline
\end{tabular}
\label{}
\end{table}


\begin{table}[H]
\caption{Numbers of $3$-potent elementsin the $Y$ Poset }
\centering
\begin{tabular}{|c|c|c||l|}\hline\hline
$n$ & $m$ & $l$  &  \multicolumn{1}{c}{  $k+1=3$}\\ [0.3ex]
\hline
3 & 3 & 3 & $270q^{22}\!+\!1620q^{21}\!+\!4626q^{20}\!+\!2808q^{19}\!+\!3648q^{18}\!+\!2808q^{17}\!+\!972q^{16}\!+\!$ \\
 &   &   &  $1026q^{15}\!+\!1044q^{14}\!+\!216q^{13}\!+\!216q^{12}\!+\!180q^{11}\!+\!108q^{10}\!+\!48q^9\!+\!54q^8\!+\!$ \\
  & & & $36q^5+3$\\
3 & 3 & 4 &  $108q^{27}\!+\!1332q^{26}\!+\!7884q^{25}\!+\!9822q^{24}\!+\!11466q^{23}\!+\!7740q^{22}\!+\!7080q^{21}\!+\!$ \\
& & & $4014q^{20}+3906q^{19}+1656q^{18}+1620q^{17}+648q^{16}+846q^{15}+288q^{14}+$ \\
&&& $198q^{13}+102q^{12}+198q^{11}+36q^{10}+42q^9+18q^8+24q^6+18q^5+3$\\ 
3 & 3  & 5 &  $360q^{32}+3780q^{31}+21090q^{30}+27000q^{29}+33150q^{28}+20304q^{27}+$ \\
& & & $23580q^{26}+14010q^{25}+10446q^{24}+6048q^{23}+6672q^{22}+3888q^{21}+ $ \\ 
& & & $2034q^{20}+1392q^{19}+1242q^{18}+594q^{17}+594q^{16}+354q^{15}+108q^{14}+$ \\
 & & & $132q^{13}+240q^{12}+18q^{11}+18q^{10}+24q^9+18q^8+30q^7+18q^5+3$\\ [1ex]
\hline
\end{tabular}
\label{}
\end{table}

\end{document}